\documentclass{article}
\usepackage{amssymb,amsmath,bm,geometry,graphics,graphicx,url,color}

\def\no{\noindent}
\def\pmatrix{\left(\begin{array}}
\def\endpmatrix{\end{array}\right)}

\def\U{{\cal U}}

\newtheorem{theo}{Theorem}
\newtheorem{lem}{Lemma}

\newtheorem{rem}{Remark}
\newtheorem{defi}{Definition}

\newtheorem{prop}{Proposition}
\newtheorem{assum}{Assumption}

\newcommand{\norm}[1]{\left\Vert#1\right\Vert}

\title{Exponential collocation methods for  \\the cubic Schr\"{o}dinger equation}

\author{Bin Wang\,
\footnote{School of Mathematical Sciences, Qufu Normal University,
Qufu 273165, P.R. China; Mathematisches Institut, University of
T\"{u}bingen, Auf der Morgenstelle 10, 72076 T\"{u}bingen, Germany.
The research is supported in part by the Alexander von Humboldt
Foundation and by the Natural Science Foundation of Shandong
Province (Outstanding Youth Foundation) under Grant ZR2017JL003.
E-mail:~{\tt wang@na.uni-tuebingen.de} } \and Xinyuan
Wu\thanks{School of Mathematical Sciences, Qufu Normal University,
Qufu 273165, P.R. China; Department of Mathematics, Nanjing
University, Nanjing 210093, P.R. China. The research is supported in
part by the National Natural Science Foundation of China under Grant
11671200. E-mail:~{\tt xywu@nju.edu.cn}} }

\begin{document}
\maketitle

\begin{abstract} In this paper we derive and analyse new exponential collocation
methods to efficiently solve the cubic Schr\"{o}dinger Cauchy
problem on a $d$-dimensional torus. Energy preservation is a key
feature of the cubic Schr\"{o}dinger equation. It is proved that the
novel methods can be of arbitrarily high order which exactly or
nearly preserve the continuous energy of the original continuous
system. The existence and uniqueness, regularity, global
convergence, nonlinear stability of the  new methods are studied in
detail. Two practical exponential collocation methods are
constructed and two numerical experiments are included. The
numerical results illustrate the efficiency of the new methods in
comparison with existing numerical methods in the literature.

\medskip
\no{\bf Keywords:} Cubic Schr\"{o}dinge equation, Energy
preservation,Exponential integrators, Collocation methods

\medskip
\no{\bf MSC:} 65P10\and 35Q55\and 65M12\and 65M70

\end{abstract}


\section{Introduction}\label{sec:introduction}
It is well known that one of the cornerstones of quantum physics is
the Schr\"{o}dinger equation. The nonlinear Schr\"{o}dinger equation
has long been used to approximately describe the dynamics of
complicated systems, such as Maxwell¡¯s equations for the
description of nonlinear optics or the equations describing surface
water waves, including rogue waves which appear from nowhere and
disappears without a trace (see, e.g.
\cite{Ablowitz81,akhmediev2009waves,benetazzo2015}). This paper is
devoted to designing and analysing novel numerical integrators to
efficiently solve the cubic Schr\"{o}dinger equation
\begin{equation}\label{sch system}
\left\{\begin{aligned} &\mathrm{i}u_{t}+\triangle
u=\lambda|u|^{2}u,\ \ (t,x)\in[0,T]\times \mathbb{T}^d,\\ &
u(0,x)=u_0(x),\ \ \ \ \ \ x\in \mathbb{T}^d,
\end{aligned}\right.
\end{equation}
where   $\mathbb{T}$ is denoted the one-dimensional torus
$\mathbb{R}/(2\pi \mathbb{Z})$ and  for the given positive integer
$d$,   $\mathbb{T}^d$ denotes the $d$-dimensional torus. It is noted
that following the researches in \cite{Cohen 12,18,Gauckler10-1}, we
only consider the cubic Schr\"{o}dinger equation in this paper. With
the  techniques developed and used in this paper, it is also
feasible to extend the novel approach to general Schr\"{o}dinger
equations. In this paper, we consider this equation as a Cauchy
problem  in time (no space discretisation is made). The solutions of
this equation have conservation of the following energy
\begin{equation}\label{H}
\begin{aligned}
H[u]=\frac{1}{2(2 \pi)^d}\int_{\mathbf{T}^d}\Big( |\nabla
u|^2+\frac{\lambda}{2}|u|^4\Big)dx.
\end{aligned}
\end{equation}
It is of great interest  to devise numerical schemes which can
conserve the continuous version  of this important invariant, and
the aim of this paper is to formulate a novel kind of exponential
integrators with  a good property of continuous energy preservation.

Schr\"{o}dinger equations frequently arise  in a wide variety of
applications including   several areas of physics, fiber optics,
quantum transport and other applied sciences (see, e.g.
\cite{Agrawal06,Faou12,jin11,Lubich-book,Sulem99}). Many  numerical
methods have been proposed for the
 integration of Schr\"{o}dinger equations, such as splitting
methods (see, e.g.
\cite{Gauckler10-1,Iserles14,Besse02,Eilinghoff16,Lubich08,McLachlan02,Thalhammer12}),
exponential-type integrators  (see, e.g. \cite{Cohen
12,Cano15,Celledoni08,Dujardin09,Ostermann17}), multi-symplectic
methods (see, e.g. \cite{Berland07,Reich00}) and other effective
methods (see, e.g. \cite{Bejenaru06,Germain09,Islas01,Kishimoto09}).

 In recent decades,  structure-preserving algorithms of differential equations have been received
 much attention  and for the  related  work,
 we   refer
the reader to
\cite{Feng2006,hairer2006,Li_Wu(na2016),wang-2016,wang2017-ANM,wu2017-JCAM,xinyuanbook2015,wu2013-book}
and references therein. It is well known that structure-preserving
algorithms are able to exactly preserve some structural properties
of the underlying continuous system. Amongst the typical subjects of
structure-preserving algorithms  are energy-preserving schemes,
which exactly preserve the energy of the underlying system. There
have been a lot of studies on this topic for Hamiltonian partial
differential equations (PDEs). In \cite{Solin-06book}, finite
element methods were introduced systematically  for numerical
solution of PDEs. The authors in \cite{Dahlby10,IMA2018} researched
discrete gradient methods  for PDEs. The work  in \cite{Celledoni12}
 investigated  the average vector field (AVF) method for discretising
Hamiltonian PDEs. Hamiltonian Boundary Value Methods (HBVMs) were
studied for the semilinear wave equation in \cite{r5} and  were
recently researched  for  nonlinear Schr\"{o}dinger equations with
wave operator in \cite{Brugnano-18}. The adapted AVF method for
Hamiltonian wave equations was analysed in
\cite{JMAA(2015)_Liu_Wu,2017_Liu_Wu}. Other related work is referred
to
\cite{Bridges06,cohen2008,Diaz07,Li_Wu(JCP2016),Liu_Iserles_Wu(2017-2),Matsuo09,CiCP(2017)_Mei_Liu_Wu,Sanz-Serna08}.
  On the
other hand,   exponential integrators have been widely introduced
and developed   for solving first-order ODEs,  and we refer the
 reader to
 \cite{hairer2006,Hochbruck2005,Hochbruck2010,Hochbruck2009,JCP2017_Mei_Wu}
for example. This kind  of methods has  also been studied  in the
numerical integration of  Schr\"{o}dinger equations (see, e.g.
\cite{Cohen 12,Cano15,Celledoni08,Dujardin09,Ostermann17}).
However, it seems that until now,   exponential integrators with a
good  continuous energy preservation for Schr\"{o}dinger equations,
have not been studied in the literature, which motivates this paper.

 With
this premise,  this paper is mainly concerned with exponential
collocation methods for solving  cubic Schr\"{o}dinger equations.
The remainder of this paper is organized as follows. We first
present some notations and preliminaries in Section \ref{sec
preliminaries}. Then the scheme of exponential collocation methods
is
 formulated  and a good continuous energy preservation is proved in Section \ref{sec:Energy-preserving schemes}.
In Section \ref{sec existence}, we  analyse the existence,
uniqueness
 and  smoothness of the methods. Section \ref{sec:regular}
pays attention to the regularity. The convergence of the methods is
studied in Section \ref{sec:convergnce} and the nonlinear stability
is discussed in Section \ref{sec:stability}. Section
\ref{sec:experiments} is devoted to constructing practical
exponential collocation methods in the light of the approach
proposed in this paper, and reporting two numerical experiments  to
demonstrate the excellent qualitative behavior of the new methods.
Section \ref{sec:conclusions} focuses on the concluding remarks.


\section{Notations and preliminaries}\label{sec preliminaries}
In this paper, we use the following notations which were presented
in \cite{Dujardin09}.

\begin{itemize}

\item We denote by $L^2(\mathbb{T}^d)$ (or simply $L^2$)
  the set of (classes of) complex functions $f$ on
$\mathbb{T}^d$  such that $\int_{\mathbb{T}^d} |f(x)|^2 dx<+\infty$,
endowed with the norm
$\norm{f}_{L^2(\mathbb{T}^d)}=\big((2\pi)^{-d}\int_{\mathbb{T}^d}
|f(x)|^2 dx\big)^{1/2}.$

\item For
$\alpha\in \mathbb{R}^+$,   denote by $H^{\alpha}(\mathbb{T}^d)$ (or
simply $H^{\alpha}$) the space of (classes of) complex functions
$f\in L^2(\mathbb{T}^d)$ such that $\sum_{k\in \mathbb{Z}^d
\backslash \{0\} }
 |\hat{f}_k |^2 |k|_2^{2\alpha}<+\infty,$   endowed with the norm
$\norm{f}_{H^{\alpha}}=\Big( |\hat{f}_0 |^2 +\sum_{k\in \mathbb{Z}^d
\backslash \{0\}}
 |\hat{f}_k |^2 |k|_2^{2\alpha}\Big)^{1/2},$  where
$\hat{f}_k=(2\pi)^{-d}\int_{\mathbb{T}^d}  f(x)e^{-\mathrm{i}kx}
dx.$ Note that $L^2 = H^0$ with the same norm.



\item For the operators $A$ from
$H^{\alpha}$ to itself, we denote $ \|A\|_{H^{\alpha}\rightarrow
H^{\alpha}}=\sup\limits_{v\in H^{\alpha}, v\neq 0}\frac{\| A
v\|_{H^{\alpha}}}{\|v\|_{H^{\alpha}}}. $
\end{itemize}

The following result  given in \cite{Dujardin09}  will be useful for
the analysis of this paper.

\begin{prop}\label{prop1} (See  \cite{Dujardin09}.)
With the above notations, if $\varphi$ is a function from
$\mathbb{C}$ to $\mathbb{C}$ bounded by $M\geq0$ on
$\textmd{i}\mathbb{R}$, then for all $h >0$ and $\alpha\geq0$, we
have
$$\norm{\varphi(\textmd{i} h \Delta)}_{H^{\alpha}\rightarrow
H^{\alpha}}\leq M.$$ For example, for all  $\alpha\geq0$, it is true
that $\norm{e^{\textmd{i}h \Delta}}_{H^{\alpha}\rightarrow
H^{\alpha}}=1.$
\end{prop}

In order to derive the novel methods, we will  use the idea of
continuous time finite element methods in
 a  generalised function  space.  To this end, we first present the
following three definitions.

\begin{defi}\label{def function space}
Define  the  generalised function space  on $[0,T]\times
\mathbb{T}^d$ as follows:
\footnote{Here we use the special notation $\textmd{span}^x$ to
express the generalised function space.}
\begin{equation*}\begin{aligned}
Y(t,x)=&\textmd{span}^x\left\{\varphi_{0}(t),\ldots,\varphi_{r-1}(t)\right\}
=\{w : w(t,x)=\sum_{j=0}^{r-1}\varphi_{j}(t) U_{j}(x),~
U_{j}(x):=\left(
               \begin{array}{c}
                 U_{j}^1(x) \\
                 U_{j}^2(x) \\
               \end{array}
             \right),\\
&\ \ \ \ \quad U_{j}^1(x), U_{j}^2(x)\ \textmd{are arbitrary
functions in}\ H^{\alpha}(\mathbb{T}^d)\},
 \end{aligned}
\end{equation*}
where $r$ is an integer satisfying  $r\geq1$ and  the functions
$\{\varphi_{j}(t)\}_{j=0}^{r-1}: [0,T]\rightarrow \mathbb{C}$ are
supposed to be linearly independent, sufficiently smooth and
integrable.

In this paper, we consider two generalised function spaces $X(t,x)$
and $Y(t,x)$ such that $w_t(t,x)\in Y(t,x)$
 for any $w(t,x)\in X(t,x)$.
 Then choose a time stepsize $h$ and define
$Y_{h}(\tau,x)$ and $X_{h}(\tau,x)$   as
\begin{equation}\label{FF-h}
Y_{h}(\tau,x)=Y(\tau h,x),\ X_{h}(\tau,x)=X(\tau h,x),
\end{equation}
 where $\tau$ is a variable satisfying $
\tau\in[0,1]$. It is noted that  throughout this paper,   the
notations $\tilde{\varphi} (\tau)$ and  $\tilde{f} (\tau,x)$ are
referred to as $ \varphi (\tau h)$ and $f(\tau h,x)$ for all the
functions, respectively.
\end{defi}


\begin{defi}
 The inner product  for the time   is defined by
\begin{equation*}\langle
\tilde{w}_{1}(\tau,x),\tilde{w}_{2}(\tau,x)\rangle_{\tau}=\int_{0}^{1}\tilde{w}_{1}(\tau,x)\cdot
\tilde{w}_{2}(\tau,x)d\tau,
\end{equation*}
where $\tilde{w}_{1}(\tau,x)$ and $\tilde{w}_{2}(\tau,x)$ are two
integrable functions for $\tau$ (scalar-valued or vector-valued)
 on $[0,1]$, and  if they are  both vector-valued functions,  `$\cdot$' denotes the entrywise multiplication
operation.
\end{defi}

\begin{defi}\label{def projection}
 A projection $\mathcal{P}_{h}\tilde{w}(\tau,x)$ onto
$Y_{h}(\tau,x)$ is defined as
\begin{equation}\label{DEF}
\langle
\tilde{v}(\tau,x),\mathcal{P}_{h}\tilde{w}(\tau,x)\rangle_{\tau}=\langle
\tilde{v}(\tau,x),\tilde{w}(\tau,x)\rangle_{\tau},\quad \text{for
any}\ \ \tilde{v}(\tau,x)\in Y_{h}(\tau,x),
\end{equation}
where    $\tilde{w}(\tau,x)$  is a continuous two-dimensional vector
function for $\tau\in[0,1]$.
\end{defi}

 With regard to  the projection operation $\mathcal{P}_{h}$,
 the following property is important.

\begin{lem}\label{proj lem}  The projection $\mathcal{P}_{h}\tilde{w}$ can be explicitly expressed as
$$ \mathcal{P}_{h}\tilde{w}(\tau,x)=\langle
P_{\tau,\sigma},\tilde{w}(\sigma,x)\rangle_{\sigma}, $$ where
$P_{\tau,\sigma}=\sum_{j=0}^{r-1}\tilde{\psi}_{j}(\tau)\tilde{\psi}_{j}(\sigma),$
 and  $\{\tilde{\psi}_j(\tau)\}_{j=0,\ldots,r-1}$   is a standard
orthonormal basis of $Y_{h}(\tau,x)$.
\end{lem}
\begin{proof}Since
$P_{h}\tilde{w}(\tau,x)\in Y_{h}(\tau,x)$,
 it can be   expressed as
$P_{h}\tilde{w}(\tau,x)=\sum_{k=0}^{r-1}\tilde{\psi}_{k}(\tau)U_{k}(x).$
By taking $\tilde{v}(\tau)=\tilde{\psi}_{l}(\tau)e_{j}\in
Y_{h}(\tau,x)$ in \eqref{DEF} for $l=0,1,\ldots,r-1$ and $j=1,2$, we
obtain
\begin{equation*}
\begin{aligned}
&\langle \tilde{\psi}_{l}(\tau)e_{j}
,\sum_{k=0}^{r-1}\tilde{\psi}_{k}(\tau)U_{k}(x)\rangle_{\tau}=
\sum_{k=0}^{r-1}\langle\tilde{\psi}_{l}(\tau)e_{j} ,\tilde{\psi}_{k}(\tau)U_{k}(x)\rangle_{\tau}\\
=&\sum_{k=0}^{r-1}\langle\tilde{\psi}_{l}(\tau),\tilde{\psi}_{k}(\tau)\rangle_{\tau}
  (e_{j}\cdot U_{k}(x)) =\langle \tilde{\psi}_{l}(\tau)e_{j}
  ,\tilde{w}(\tau,x)\rangle_{\tau},
\end{aligned}
\end{equation*}
which gives
$\sum_{k=0}^{r-1}\langle\tilde{\psi}_{l}(\tau),\tilde{\psi}_{k}(\tau)\rangle_{\tau}
  U_{k} (x)=\langle
\tilde{\psi}_{l}(\tau),\tilde{w}(\tau,x)\rangle_{\tau}.$ By the
standard orthonormal basis $\{\tilde{\psi}_j(t)\}_{j=0,\ldots,r-1}$,
this result can be formulated as
\begin{equation*}
  \left(\begin{array}{c}U_{0}(x)\\
\vdots\\U_{r-1}(x)\end{array}\right)= \left(\begin{array}{c}\langle
\tilde{\psi}_{0}(\tau),\tilde{w}(\tau,x)\rangle_{\tau}\\ \vdots\\
\langle\tilde{\psi}_{r-1}(\tau),\tilde{w}(\tau,x)\rangle_{\tau}\end{array}\right).
\end{equation*}
Then one has
\begin{equation*}
\begin{aligned}
&P_{h}\tilde{w}(\tau,x)=(\tilde{\psi}_{0}(\tau),\ldots,\tilde{\psi}_{r-1}(\tau))  \otimes I_{2\times 2}\left(\begin{array}{c}U_{0}(x)\\
\vdots\\U_{r-1}(x)\end{array}\right)\\
=&(\tilde{\psi}_{0}(\tau),\ldots,\tilde{\psi}_{r-1}(\tau)) \otimes
I_{2\times 2}\left(\begin{array}{c}\langle
\tilde{\psi}_{0}(\tau),\tilde{w}(\tau,x)\rangle_{\tau}\\ \vdots\\
\langle\tilde{\psi}_{r-1}(\tau),\tilde{w}(\tau,x)\rangle_{\tau}\end{array}\right)=\langle
P_{\tau,\sigma},\tilde{w}(\sigma,x)\rangle_{\sigma},
\end{aligned}
\end{equation*}
which proves the result.
 \hfill  $\square$
\end{proof}

\begin{rem}
It is noted that the above three definitions and Lemma \ref{proj
lem} can be considered as the generalised version of those presented
in \cite{Li_Wu(na2016)}.   We also remark  that one can make
different choices of  $Y_h(\tau,x)$ and $X_h(\tau,x)$, which will
produce different
  methods by taking the  methodology proposed in this paper.
\end{rem}
\section{Formulation of  exponential collocation methods}\label{sec:Energy-preserving schemes}

Define  the linear differential operator $\mathcal{A}$ by
\begin{equation*}\label{operator}
(\mathcal{A}u)(x,t)= \Delta u(x,t)
\end{equation*}
and let  $f(u)=-\lambda|u|^{2}u$.
 Then the system \eqref{sch system} is identical to
\begin{equation}\label{sch ode}
\begin{aligned}
&u_{t}=\mathrm{i}\mathcal{A}u+\mathrm{i}f(u), \ \  \ \
u(0,x)=u_0(x).
\end{aligned}
\end{equation}
The solutions of this equation satisfy Duhamel's formula
\begin{equation}\label{Duhamel formu}
\begin{aligned}
u(t_n+h)=e^{\mathrm{i}h \mathcal{A}}u(t_n)+\mathrm{i}   \int_{0}^h
e^{\mathrm{i}(h-\xi) \mathcal{A}}f(u(t_n+\xi))d\xi.
\end{aligned}
\end{equation}
Let $u = p + \textmd{i}q,$ and then the equation \eqref{sch system}
can be rewritten as a pair of real initial-value problems
\begin{equation}\label{rea sd}
\begin{aligned}
&\left(
  \begin{array}{c}
   p_t \\
  q_t\\
  \end{array}
\right)=
 \begin{aligned}
 \left(
   \begin{array}{cc}
     0 & -\mathcal{A} \\
     \mathcal{A} & 0 \\
   \end{array}
 \right)\left(
  \begin{array}{c}
   p \\
   q \\
  \end{array}
\right)+\left(
          \begin{array}{c}
           \lambda(p^{2}+ q^{2})  q \\
            -\lambda( p^{2}+ q^{2})  p \\
          \end{array}
        \right)
\end{aligned},\ \
  \left(
                       \begin{array}{c}
                         p(0,x) \\
                         q(0,x) \\
                       \end{array}
                     \right)=\left(
                               \begin{array}{c}
                                 \textmd{Re}(u_0(x)) \\
                                 \textmd{Im}(u_0(x)) \\
                               \end{array}
                             \right).
\end{aligned}
\end{equation}
In this case, the energy  of this system is expressed by
 \begin{equation}\label{rea H}
\begin{aligned}
&\mathcal{H}(p,q) =\frac{1}{2(2 \pi)^d}\int_{\mathbf{T}^d}\Big(
|\nabla p|^2+ |\nabla
q|^2+\frac{\lambda}{2}(p^2+q^2)^2 \Big)dx.\\
\end{aligned}
\end{equation}
Accordingly, the system
 \eqref{rea sd} can be  formulated  as the
following infinite-dimensional real Hamiltonian system
\begin{equation}\label{new sch}
\begin{aligned}
y_t=\mathcal{K}y+g(y)=J^{-1}\frac{\delta \mathcal{H}}{\delta y},\ \
\ \ y_0(x)= \left(
                       \begin{array}{c}
                         p(0,x) \\
                         q(0,x) \\
                       \end{array}
                     \right),
\end{aligned}
\end{equation}
where $ \mathcal{K}=\left(
   \begin{array}{cc}
     0 & -\mathcal{A} \\
     \mathcal{A} & 0 \\
   \end{array}
 \right),$
$J=\left(
      \begin{array}{cc}
        0 & -1 \\
        1 & 0 \\
      \end{array}
    \right),\ \ y=\left(
                             \begin{array}{c}
                               p \\
                               q \\
                             \end{array}
                           \right),$ and $g(y)=\left(
          \begin{array}{c}
           \lambda(p^{2}+ q^{2})  q \\
            -\lambda( p^{2}+ q^{2})  p \\
          \end{array}
        \right).$

With the preliminaries described above,   we first derive the
exponential collocation methods for solving the real-valued equation
\eqref{new sch} and then present the methods for solving the cubic
Schr\"{o}dinger equation \eqref{sch system}.

We consider a function $\tilde{z}(\tau,x)$ with
$\tilde{z}(0,x)=y_{0}(x)$, satisfying that
\begin{equation}\label{projection}
 \begin{aligned}
\tilde{z}_{t}(\tau,x)&=\mathcal{K} \tilde{z}(\tau,x)+\mathcal{P}_{h}
g(\tilde{z}(\tau,x)) =\mathcal{K} \tilde{z}(\tau,x)+\langle
P_{\tau,\sigma},g(\tilde{z}(\sigma,x))\rangle_{\sigma}.
\end{aligned}
\end{equation}
 Applying     Duhamel's formula  \eqref{Duhamel formu} to
\eqref{projection} leads to
\begin{equation*}
 \begin{aligned} &\tilde{z}(\tau,x)
=e^{\tau h \mathcal{K}}y_{0}(x)+\tau h    \int_{0}^1 e^{(1-\xi)\tau
h \mathcal{K}}\langle P_{\xi
\tau,\sigma},g(\tilde{z}(\sigma,x))\rangle_{\sigma}d\xi\\
 &=e^{\tau h \mathcal{K}}y_{0}(x)+\tau h  \int_{0}^1 e^{(1-\xi)\tau h
\mathcal{K}} \int_{0}^1 \sum_{j=0}^{r-1}\tilde{\psi}_{j}(\xi
\tau)\tilde{\psi}_{j}(\sigma)g(\tilde{z}(\sigma,x)) d\sigma d\xi\\
&=e^{\tau h \mathcal{K}}y_{0}(x)+\tau h  \int_{0}^1 \sum_{j=0}^{r-1}
 \int_{0}^1 e^{(1-\xi)\tau h \mathcal{K}}\tilde{\psi}_{j}(\xi \tau)d\xi \tilde{\psi}_{j}(\sigma)g(\tilde{z}(\sigma,x))
d\sigma.
\end{aligned}
\end{equation*}
This yields the following definition of exponential collocation
methods.
\begin{defi}\label{def ECMr}
An exponential collocation method  for  solving the real Hamiltonian
initial-value problem \eqref{new sch} is defined by
\begin{equation}\label{ECMr}
\left\{\begin{aligned} &\tilde{z}(\tau,x)=e^{\tau h
\mathcal{K}}y_{0}(x)+\tau h \int_{0}^1
\bar{A}_{\tau,\sigma}(\mathcal{K}) g(\tilde{z}(\sigma,x))
d\sigma,\quad 0\leq\tau\leq1,\\
 & y_{1}(x)=\tilde{z}(1,x),
\end{aligned}\right.
\end{equation}
where $h$ is a time stepsize and
\begin{equation}\label{Aexplicit}
\bar{A}_{\tau,\sigma}(\mathcal{K})= \int_{0}^1 e^{(1-\xi)\tau h
\mathcal{K}}P_{\xi \tau,\sigma}d\xi=\sum_{i=0}^{r-1}
 \int_{0}^1 e^{(1-\xi)\tau h \mathcal{K}}\tilde{\psi}_{i}(\xi \tau)d\xi
\tilde{\psi}_{i}(\sigma).
\end{equation}

\end{defi}

 \begin{theo}\label{EP}
If $\tilde{z}(\tau,x) \in X_h(\tau,x)$, the continuous energy
$\mathcal{H}$ determined by \eqref{rea H} is preserved exactly by
the method \eqref{ECMr}, i.e.,
$$\mathcal{H}(y_{1}(x))=\mathcal{H}(y_{0}(x)).$$
If $\tilde{z}(\tau,x) \notin X_h(\tau,x)$, the exponential
collocation method \eqref{ECMr} approximately preserves the
continuous energy $\mathcal{H}$ with the following accuracy
$$\mathcal{H}(y_{1}(x))=\mathcal{H}(y_{0}(x))+\mathcal{O}(h^{2r+1}).$$
\end{theo}
\begin{proof}We first prove the first statement. Since $\tilde{z}(\tau,x)\in X_{h}(\tau,x)$, we obtain that
$\tilde{z}_{t}(\tau,x)\in Y_{h}(\tau,x)$. Therefore, it follows from
\eqref{projection} that
\begin{equation}\label{inner}
\begin{aligned}
&\int_{0}^{1}\tilde{z}_{t}(\tau,x)^{\intercal}J^{\intercal}\tilde{z}_{t}(\tau,x)d\tau
=\int_{0}^{1}\tilde{z}_{t}(\tau,x)^{\intercal}J^{\intercal}\big(\mathcal{K}\tilde{z}(\tau,x)+\mathcal{P}_{h}g(\tilde{z}(\tau,x))\big)d\tau\\
 =& \int_{0}^{1}\tilde{z}_{t}(\tau,x)^{\intercal}J^{\intercal}\big(\mathcal{K}\tilde{z}(\tau,x)+g(\tilde{z}(\tau,x))\big)d\tau.\\
\end{aligned}
\end{equation}
 Because  $J$  is skew symmetric,      we have
$$0=\int_{0}^{1}\tilde{z}_{t}(\tau,x)^{\intercal}J^{\intercal}\tilde{z}_{t}(\tau,x)d\tau
=\int_{0}^{1}\tilde{z}_{t}(\tau,x)^{\intercal}J\big(\mathcal{K}
\tilde{z}(\tau,x)+g(\tilde{z}(\tau,x))\big)d\tau.$$ Therefore, it is
true that
\begin{equation*}
\begin{aligned}
&\mathcal{H}(y_{1}(x))-\mathcal{H}(y_{0}(x))=\int_{0}^{1}\frac{\partial}{\partial\tau}\mathcal{H}(\tilde{z}(\tau,x))d\tau
=h\int_{0}^{1}\tilde{z}_{t}(\tau,x)^{\intercal}\frac{\delta \mathcal{H}(\tilde{z})}{\delta \tilde{z}}  d\tau\\
&
=h\int_{0}^{1}\tilde{z}_{t}(\tau,x)^{\intercal}J\big(\mathcal{K}\tilde{z}(\tau,x)+g(\tilde{z}(\tau,x))\big)d\tau=h\cdot0=
0.
\end{aligned}
\end{equation*}

 For the second statement, according to the above
analysis, we obtain
\begin{equation*}
\begin{aligned}
&\mathcal{H}(y_{1}(x))-\mathcal{H}(y_{0}(x))
=h\int_{0}^{1}\tilde{z}_{t}(\tau,x)^{\intercal}J\big(\mathcal{K}\tilde{z}(\tau,x)+g(\tilde{z}(\tau,x))\big)d\tau\\
=&h\int_{0}^{1}\tilde{z}_{t}(\tau,x)^{\intercal}J\big(\mathcal{K}\tilde{z}(\tau,x)+\mathcal{P}_{h}g(\tilde{z}(\tau,x))
+g(\tilde{z}(\tau,x))-\mathcal{P}_{h}g(\tilde{z}(\tau,x))\big)d\tau\\
=&h\int_{0}^{1}\tilde{z}_{t}(\tau,x)^{\intercal}J\tilde{z}_{t}(\tau,x)+
h\int_{0}^{1}\tilde{z}_{t}(\tau,x)^{\intercal}J\big(g(\tilde{z}(\tau,x))-\mathcal{P}_{h}g(\tilde{z}(\tau,x))\big)d\tau\\
=&
h\int_{0}^{1}\tilde{z}_{t}(\tau,x)^{\intercal}J\big(g(\tilde{z}(\tau,x))-\mathcal{P}_{h}g(\tilde{z}(\tau,x))\big)d\tau.
\end{aligned}
\end{equation*}
From the results of Lemmas \ref{lemmare}-\ref{lemmaph} which are
proved in Section \ref{sec:regular}, it follows that
$\tilde{z}_{t}(\tau,x)=\mathcal{P}_{h}\tilde{z}_{t}(\tau,x)+\mathcal{O}(h^{r})$
and
$g(\tilde{z}(\tau,x))-\mathcal{P}_{h}g(\tilde{z}(\tau,x))=\mathcal{O}(h^{r})$.
Therefore, one arrives at
\begin{equation*}
\begin{aligned}
&\mathcal{H}(y_{1}(x))-\mathcal{H}(y_{0}(x)) \\=&
h\int_{0}^{1}\big(\mathcal{P}_{h}\tilde{z}_{t}(\tau,x)+\mathcal{O}(h^{r})\big)^{\intercal}J\big(g(\tilde{z}(\tau,x))-\mathcal{P}_{h}g(\tilde{z}(\tau,x))\big)d\tau\\
 =&
h\int_{0}^{1}\big(\mathcal{P}_{h}\tilde{z}_{t}(\tau,x)\big)^{\intercal}J\big(g(\tilde{z}(\tau,x))-\mathcal{P}_{h}g(\tilde{z}(\tau,x))\big)d\tau+\mathcal{O}(h^{2r+1})\\
 =&
h\int_{0}^{1}\big(\mathcal{P}_{h}\tilde{z}_{t}(\tau,x)\big)^{\intercal}J\big(g(\tilde{z}(\tau,x))-g(\tilde{z}(\tau,x))\big)d\tau+\mathcal{O}(h^{2r+1})=\mathcal{O}(h^{2r+1}).
\end{aligned}
\end{equation*}

 \hfill  $\square$
 \end{proof}

In terms of the variables appearing in \eqref{sch system} instead of
\eqref{new sch},  we can rewrite Definition \ref{def ECMr} for the
Schr\"{o}dinger equation \eqref{sch system} as follows.

 \begin{defi}\label{defECMr}An exponential collocation methods  (denoted as ECMr) with a time stepsize $h$
for the cubic Schr\"{o}dinger equation  \eqref{sch system} is
defined by
\begin{equation}\label{ECM}
\left\{\begin{aligned} &\tilde{u}(\tau,x)=e^{\tau h
\mathrm{i}\mathcal{A}}u_{0}(x)+ \tau h \int_{0}^1
\bar{A}_{\tau,\sigma}( \mathrm{i} \mathcal{A})
\mathrm{i}f(\tilde{u}(\sigma,x)) d\sigma,\quad 0\leq\tau\leq1,\\
& u_{1}(x)=\tilde{u}(1,x),
\end{aligned}\right.
\end{equation}
where
\begin{equation}\label{Aexplicit}
\bar{A}_{\tau,\sigma}(\mathrm{i} \mathcal{A})= \int_{0}^1
e^{(1-\xi)\tau h
 \mathrm{i} \mathcal{A}}P_{\xi \tau,\sigma}d\xi=\sum_{j=0}^{r-1}
 \int_{0}^1 e^{(1-\xi)\tau h  \mathrm{i} \mathcal{A}}\tilde{\psi}_{j}(\xi \tau)d\xi
\tilde{\psi}_{j}(\sigma).
\end{equation}
\end{defi}

\begin{rem}It is noted that this novel method shares
the advantages of exponential integrators  and collocation methods.
\end{rem}

\section{The existence, uniqueness
 and  smoothness}\label{sec existence}
In the remainder of this paper, we use the following assumptions on
the exact solution $u$ of \eqref{sch system} and on the
non-linearity $f$.

\begin{assum}\label{ass1} (See  \cite{Dujardin09}.)
It is assumed that the    Schr\"{o}dinger equation \eqref{sch
system} admits an exact solution $u:[0, T ] \rightarrow
H^{\alpha}(\mathbb{T}^d)$ which is sufficiently smooth. In
particular, there exists $R> 0$ such that $u\in B(\bar{u}_{0},R)$
for all $t \in [0, T ]$, where $$B(\bar{u}_{0},R)=\left\{u\in
H^{\alpha}(\mathbb{T}^d): ||u-\bar{u}_{0}||_{H^{\alpha}}\leq
R\right\} $$ and $\bar{u}_{0}=e^{\tau h
\mathrm{i}\mathcal{A}}u_{0}(x)$.
\end{assum}

\begin{assum}\label{ass2}(See  \cite{Dujardin09}.)
We assume that the mapping $f : [0, T ] \rightarrow
H^{\alpha}(\mathbb{T}^d)$ is sufficiently smooth.
\end{assum}

\begin{assum}\label{ass3}
The $r$th-order derivative of $f$     is assumed that $f^{(r)}\in
C^1([0,T],H^{\alpha}(\mathbb{T}^d))$.  And we denote $
D_{n}=\max_{u\in B(\bar{u}_{0},R)}||f^{(n)}(u)||_{H^{\alpha}}$ for $
n=0,\ldots,r. $
\end{assum}

\begin{assum}\label{ass4}
The initial value function $u_0(x)$ is assumed to be regular enough
such that $\norm{\mathcal{A}^lu_{0}(x)}_{H^{\alpha}}$ are uniformly
bounded for $l=1,2,\ldots.$
\end{assum}


Note that the first three assumptions are fulfilled in the case of
the cubic non-linear Schr\"{o}dinger equation (see \cite{Alinhac91}
for example, Chapter II, Proposition 2.2), at least when
$\alpha>d/2$ and $\alpha-d/2\not\in \mathbb{N}$.

According to Proposition  \ref{prop1}, one gets that the
coefficients $e^{\tau h  \mathrm{i} \mathcal{A}}$ and
$\bar{A}_{\tau,\sigma}( \mathrm{i} \mathcal{A})$ of our methods for
$0\leq\tau\leq1$ and $0\leq\sigma\leq1$   are uniformly bounded.
 Hence, we let
\begin{equation}\label{MC}
M_{k}=\max_{\tau,\sigma,h\in[0,1]}\norm{\frac{\partial^{k}\bar{A}_{\tau,\sigma}}{\partial
h^{k}}}_{H^{\alpha}\rightarrow H^{\alpha}},\
C_{k}=\max_{\tau,h\in[0,1]}\norm{\frac{\partial^{k} e^{\tau h
\mathrm{i}\mathcal{A}}}{\partial h^{k}}u_{0}(x)}_{H^{\alpha}}, \
k=0,1,\ldots.
\end{equation}
It is noted that these bounds are finite when the operators are
applied over certain regular enough functions.

\begin{theo}\label{eus} Under the above assumptions,
 if the time stepsize $h$ satisfies
\begin{equation}\label{cond2} 0\leq
h\leq
\kappa<\min\left\{\frac{1}{M_{0}D_{1}},\frac{R}{M_{0}D_{0}},1\right\},
\end{equation}
 then the ECMr
method \eqref{ECM} admits a unique solution $\tilde{u}(\tau,x)$
which  smoothly depends on  $h$.
\end{theo}
\begin{proof}
By setting $\tilde{u}_{0}(\tau,x)=u_{0}(x)$ and defining
\begin{equation}\label{recur}
\tilde{u}_{n+1}(\tau,x)= e^{\tau h
\mathrm{i}\mathcal{A}}u_{0}(x)+\tau h \int_{0}^1
\bar{A}_{\tau,\sigma}(\mathrm{i}\mathcal{A}) \mathrm{i}
g(\tilde{u}_n(\sigma,x)) d\sigma, \quad n=0,1,\ldots,
\end{equation}
we get a function series $\{\tilde{u}_{n}(\tau,x)\}_{n=0}^{\infty}.$
If $\left\{\tilde{u}_{n}(\tau,x)\right\}_{n=0}^{\infty}$ is
uniformly convergent,
$\lim\limits_{n\to\infty}\tilde{u}_{n}(\tau,x)$ is a solution  of
 the ECMr
method \eqref{ECM}.

By induction,  it follows from \eqref{cond2} and \eqref{recur} that
$ \tilde{u}_{n}(\tau,x)\in B(\bar{u}_{0},R)$ for $n=0,1,\ldots.$
Using \eqref{recur}, we have
\begin{equation*}
\begin{aligned}
&||\tilde{u}_{n+1}(\tau,x)-\tilde{u}_{n}(\tau,x)||_{H^{\alpha}}\leq \tau h\int_{0}^{1}M_{0}D_{1}||\tilde{u}_{n}(\sigma,x)-\tilde{u}_{n-1}(\sigma,x)||_{H^{\alpha}}d\sigma\\
& \leq  h\int_{0}^{1}M_{0}D_{1}||\tilde{u}_{n}(\sigma,x)-\tilde{u}_{n-1}(\sigma,x)||_{H^{\alpha}}d\sigma \leq\beta||\tilde{u}_{n}-\tilde{u}_{n-1}||_{c},\quad{\beta=\kappa M_{0}D_{1},}\\
\end{aligned}
\end{equation*}
where $||\cdot||_{c}$ is the maximum norm for continuous functions
defined as $
||w||_{c}=\max_{\tau\in[0,1]}||w(\tau,x)||_{H^{\alpha}}$ for
  a continuous   function $w(\tau,x)$ with $\tau $ on
$[0,1]$.  Hence, one arrives at
\begin{equation*}
||\tilde{u}_{n+1}(\tau,x)-\tilde{u}_{n}(\tau,x)||_{c}\leq\beta||\tilde{u}_{n}(\tau,x)-\tilde{u}_{n-1}(\tau,x)||_{c},
\end{equation*}
and
\begin{equation}\label{recur3}
||\tilde{u}_{n+1}(\tau,x)-\tilde{u}_{n}(\tau,x)||_{c}\leq\beta^{n}||\tilde{u}_{1}(\tau,x)-u_{0}(x)||_{c}\leq\beta^{n}R,\quad
n=0,1,\ldots.
\end{equation}
  Weierstrass $M$-test and the fact that $\beta<1$  yield  the
 uniformly convergence of
$\sum_{n=0}^{\infty}(\tilde{u}_{n+1}(\tau,x)-\tilde{u}_{n}(\tau,x))$.

If $\tilde{v}(\tau,x)$ is  another solution of the method, then it
is obtained that
 \begin{equation*}
||\tilde{u}(\tau,x)-\tilde{v}(\tau,x)||_{H^{\alpha}}\leq
h\int_{0}^{1}||\bar{A}_{\tau,\sigma}(\mathrm{i}\mathcal{A})\mathrm{i}
\big(g(\tilde{u}(\sigma,x))-g(\tilde{v}(\sigma,x))\big)||_{H^{\alpha}}d\sigma\leq\beta||\tilde{u}-\tilde{v}||_{c},
\end{equation*}
and $
\norm{\tilde{u}-\tilde{v}}_{c}\leq\beta||\tilde{u}-\tilde{v}||_{c}.
$ This leads to   $||\tilde{u}-\tilde{v}||_{c}=0$  and then
$\tilde{u}=\tilde{v}$.

In order to prove that  $\tilde{u}(\tau,x)$  is smoothly dependent
of $h$, we need  to prove that the series
$\left\{\frac{\partial^{k}\tilde{u}_{n}}{\partial
h^{k}}(\tau,x)\right\}_{n=0}^{\infty}$  is  uniformly convergent for
$k\geq1$. It follows from   \eqref{recur} that
\begin{equation}\label{recur2}
\begin{aligned}
\frac{\partial \tilde{u}_{n+1}}{\partial h}(\tau,x)=&\tau
\mathrm{i}\mathcal{A}e^{\tau h \mathrm{i}\mathcal{A}}u_0+\tau
\int_{0}^{1}\Big(\bar{A}_{\tau,\sigma}(\mathrm{i}\mathcal{A})+h\frac{\partial
\bar{A}_{\tau,\sigma}}{\partial
h}\Big)\mathrm{i} g(\tilde{u}_{n}(\sigma,x))d\sigma \\
&+\tau
h\int_{0}^{1}\bar{A}_{\tau,\sigma}(\mathrm{i}\mathcal{A})\mathrm{i}
g^{(1)}(\tilde{u}_{n}(\sigma,x))\frac{\partial
\tilde{u}_{n}}{\partial h}(\sigma,x)d\sigma, \end{aligned}
\end{equation}
which yields
\begin{equation*}
\norm{\frac{\partial \tilde{u}_{n+1}}{\partial
h}}_{c}\leq\alpha+\beta\norm{\frac{\partial \tilde{u}_{n}}{\partial
h}}_{c},\quad \alpha=C_1+(M_{0}+\kappa M_{1})D_{0}.
\end{equation*}
Therefore,  it is easy to show that $\left\{\frac{\partial
\tilde{u}_{n}}{\partial h}(\tau,x)\right\}_{n=0}^{\infty}$ is
uniformly bounded:
\begin{equation}\label{bound}
\norm{\frac{\partial \tilde{u}_{n}}{\partial
h}}_{c}\leq\alpha(1+\beta+\ldots+\beta^{n-1})\leq\frac{\alpha}{1-\beta}=C^{*},\quad
n=0,1,\ldots.
\end{equation}
Moreover, in the light of  \eqref{recur3}--\eqref{bound}, one
obtains
\begin{equation*}
\begin{aligned}
&\norm{\frac{\partial \tilde{u}_{n+1}}{\partial h}-\frac{\partial
\tilde{u}_{n}}{\partial h}}_{c} \leq
 \tau \int_{0}^{1}(M_{0}+hM_{1})\norm{g(\tilde{u}_{n}(\sigma,x))-g(\tilde{u}_{n-1}(\sigma,x))}_{H^{\alpha}}d\sigma\\
&+\tau
h\int_{0}^{1}M_{0}\Big(\norm{\big(g^{(1)}(\tilde{u}_{n}(\sigma,x))-g^{(1)}(\tilde{u}_{n-1}(\sigma,x))\big)\frac{\partial
\tilde{u}_{n}}{\partial h}(\sigma,x)}_{H^{\alpha}}\\&
+\norm{g^{(1)}(\tilde{u}_{n-1}(\sigma,x))\Big(\frac{\partial
\tilde{u}_{n}}{\partial h}(\sigma,x)-\frac{\partial \tilde{u}_{n-1}}
{\partial h}(\sigma,x)\Big)}_{H^{\alpha}}\Big)d\sigma\\
&\leq\gamma\beta^{n-1}+\beta\norm{\frac{\partial \tilde{u}_{n}}{\partial h}-\frac{\partial \tilde{u}_{n-1}}{\partial h}}_{c},\\
\end{aligned}
\end{equation*}
where $\gamma=(M_{0}D_{1}+\kappa M_{1}D_{1}+\kappa
M_{0}L_{2}C^{*})R,$
 and $L_{2}$ is a constant satisfying
\begin{equation*}
||g^{(1)}(y)-g^{(1)}(z)||_{H^{\alpha}}\leq
L_{2}||y-z||_{H^{\alpha}}\quad\text{for\ \ $y,z\in
B(\bar{y}_{0},R)$}.
\end{equation*}
By induction we get
\begin{equation*}
\norm{\frac{\partial \tilde{u}_{n+1}}{\partial h}-\frac{\partial
\tilde{u}_{n}}{\partial h}}_{c}\leq
n\gamma\beta^{n-1}+\beta^{n}C^{*},\quad n=1,2,\ldots.
\end{equation*}
Thus,  $\left\{\frac{\partial \tilde{u}_{n}}{\partial
h}(\tau,x)\right\}_{n=0}^{\infty}$ is uniformly convergent.

 In a similar way,  it can be proved  that other function series
$\left\{\frac{\partial^{k}\tilde{u}_{n}}{\partial
h^{k}}(\tau,x)\right\}_{n=0}^{\infty}$ for $k\geq2$ are uniformly
convergent. Therefore, $\tilde{u}(\tau,x)$ is  smoothly dependent on
$h$.
 \hfill  $\square$
\end{proof}

\section{$h$-dependent regularity of the methods}\label{sec:regular}
 In this section, we study the regularity of the methods. In this paper,  a function $w(\tau,x)$ is called as
$h$-dependent regular if it can be expanded as
\begin{equation*}
w(\tau,x)=\sum_{n=0}^{r-1}w^{[n]}(\tau,x)h^{n}+\mathcal{O}(h^{r}),
\end{equation*}
where
$$w^{[n]}(\tau,x)=\frac{1}{n!}\frac{\partial^{n}w(\tau,x)}{\partial
h^{n}}|_{h=0}\in S_n:=
\textmd{span}^x\left\{1,\tau,\ldots,\tau^n\right\}.$$

For the $P_{\tau,\sigma}$ given in Lemma \ref{proj lem}, we need the
following property.
\begin{prop}\label{prop2}
Assume that the Taylor expansion of $P_{\tau,\sigma}$ with respect
to $h$ at zero is
\begin{equation}\label{Pexpand}
P_{\tau,\sigma}=\sum_{n=0}^{r-1}P_{\tau,\sigma}^{[n]}h^{n}+\mathcal{O}(h^{r}).
\end{equation}
Then the coefficients $P_{\tau,\sigma}^{[n]}$ satisfy
\begin{equation*}
\langle P_{\tau,\sigma}^{[n]},g_{m}(\sigma,x)\rangle_{\sigma}=
\left\{\begin{aligned}
&g_{m}(\tau,x) ,{\quad n=0,\quad m=r-1,}\\
&0,\quad{n=1,\ldots,r-1,\quad m=r-1-n,}\\
\end{aligned}\right.
\end{equation*}
for any $g_{m}(\tau,x) \in
\textmd{span}^x\left\{1,\tau,\ldots,\tau^m\right\}$.
\end{prop}
\begin{proof}According to the analysis in \cite{Li_Wu(na2016)}, we have that
\begin{equation*}
\langle P_{\tau,\sigma}^{[n]},\varphi_{m}(\sigma)\rangle_{\sigma}=
\left\{\begin{aligned}
&\varphi_{m}(\tau) ,{\quad n=0,\quad m=r-1,}\\
&0,\quad{n=1,\ldots,r-1,\quad m=r-1-n,}\\
\end{aligned}\right.
\end{equation*}
for any $\varphi_{m}\in P_{m}([0,1]),$  where  $P_{m}([0,1])$
consists of polynomials of degrees $\leq m$ on $[0,1]$. From the
definition of the space
$\text{span}\left\{1,\tau,\ldots,\tau^m\right\}$, it is clear that
the result is true.
 \hfill  $\square$
\end{proof}

\begin{lem}\label{lemmare}
The ECMr  method  \eqref{ECM} gives an  $h$-dependent regular
 function $\tilde{u}(\tau,x)$.
\end{lem}
\begin{proof}
By the result given in Theorem \ref{eus}, we know that
$\tilde{u}(\tau,x)$ can be expanded with respect to $h$ at zero as $
\tilde{u}(\tau,x)=\sum_{m=0}^{r-1}\tilde{u}^{[m]}(\tau,x)h^{m}+\mathcal{O}(h^{r}).
$ From the definition of $\bar{A}_{\tau,\sigma}(\mathrm{i}
\mathcal{A})$ given in \eqref{Aexplicit}, it follows that
$\bar{A}_{\tau,\sigma}(\mathrm{i} \mathcal{A})$ is $h$-dependent
regular, i.e.,
\begin{equation*}
\begin{aligned}
&\bar{A}_{\tau,\sigma}(\mathrm{i} \mathcal{A})
 =\sum_{k=0}^{r-1} \bar{A}^{[k]}_{\tau,\sigma}(\mathrm{i}
 \mathcal{A})
 h^k+\mathcal{O}(h^{r}),
\end{aligned}
\end{equation*}
where $\bar{A}^{[k]}_{\tau,\sigma}(\mathrm{i} \mathcal{A})\in S_k.$
Denote $\delta=\tilde{u}(\tau,x)-u_{0}(x)$  and then we have
\begin{equation*}
\delta=
\tilde{u}^{[0]}(\tau,x)-u_{0}(x)+\mathcal{O}(h)=u_{0}(x)-u_{0}(x)+\mathcal{O}(h)=\mathcal{O}(h).
\end{equation*}
We now return to the scheme \eqref{ECM} of  ECMr method. Expanding
$f(\tilde{u}(\tau,x))$ at $u_{0}(x)$ and inserting  the above
equalities into the scheme, one gets
\begin{equation}\label{comp}
\begin{aligned}
\sum_{m=0}^{r-1}\tilde{u}^{[m]}(\tau,x)h^{m}=&\sum_{m=0}^{r-1}
\frac{(\tau   \mathrm{i}  \mathcal{A})^mu_{0}(x)}{m!} h^{m}
+\mathrm{i}\tau h\int_{0}^{1}\sum_{k=0}^{r-1}
\bar{A}^{[k]}_{\tau,\sigma}(\mathrm{i} \mathcal{A})
 h^k\\
 &\sum_{n=0}^{r-1}\frac{1}{n!}f^{(n)}(u_{0}(x))
 (\underbrace{\delta,\ldots,\delta}_{n-fold})d\sigma+\mathcal{O}(h^{r}).
 \end{aligned}
\end{equation}
In what follows, it is  needed  only  to prove by induction that
$$\tilde{u}^{[m]}(\tau,x)\in S_m\ \ \ \textmd{for}\  \ \ m=0,\ldots,r-1.$$

Firstly, $\tilde{u}^{[0]}(\tau,x)=u_{0}(x)\in S_0$. Assume that
$\tilde{u}^{[n]}(\tau,x)\in S_{n}$ for $n=0,1,\ldots,m$. Comparing
the coefficients of $h^{m+1}$ on both sides of \eqref{comp} yields
\begin{equation*}
\begin{aligned}
&\tilde{u}^{[m+1]}(\tau,x)=\frac{(\tau  \mathrm{i}
\mathcal{A})^{m+1} }{(m+1)!} u_{0}(x) +\mathrm{i}\tau\sum_{k+n=m}
 \int_{0}^{1}
\bar{A}^{[k]}_{\tau,\sigma}(\mathrm{i} \mathcal{A})
\frac{1}{n!}f^{(n)}(u_{0}(x)) d\sigma,
\end{aligned}
\end{equation*}
 which confirms that $
\tilde{u}^{[m+1]}(\tau,x)  \in S_{m+1}.$
 \hfill  $\square$
\end{proof}

About  $h$-dependent  regular functions,  we have the following
property which will be used in the  remainder  of this paper.

\begin{lem}\label{lemmaph}
Given a regular function $w$ and an $h$-independent  sufficiently
smooth  function $g$, the composition (if exists) is regular.
Moreover,  the difference between $w$ and its projection satisfies
\begin{equation*}
\mathcal{P}_{h}w(\tau,x)-w(\tau,x)=\mathcal{O}(h^{r}).
\end{equation*}
\end{lem}
\begin{proof} For the first result,
assume that $
g(w(\tau,x))=\sum_{n=0}^{r-1}p^{[n]}(\tau,x)h^{n}+\mathcal{O}(h^{r}).
$ Then by differentiating $g(w(\tau,x))$ with respect to $h$ at zero
iteratively and using
$$p^{[n]}(\tau,x)=\frac{1}{n!}\frac{\partial^{n}g(w(\tau,x))}{\partial h^{n}}|_{h=0},
\quad\text{ $\frac{\partial^{n}w(\tau,x)}{\partial h^{n}}|_{h=0}\in
S_n$},\quad {n=0,1,\ldots,r-1,}$$ it can be observed that
$p^{[n]}(\tau,x)\in S_n$   for $n=0,1,\ldots,r-1$.

As for the second statement,  in terms of Proposition \ref{prop2},
we have
\begin{equation*}
\begin{aligned}
&\ \ \ \ \mathcal{P}_{h}w(\tau,x)-w(\tau,x)=\langle P_{\tau,\sigma},w(\sigma,x) \rangle_{\sigma}-w(\tau,x)\\
&=\langle\sum_{n=0}^{r-1}P_{\tau,\sigma}^{[n]}h^{n},\sum_{k=0}^{r-1}w^{[k]}(\sigma,x)h^{k}\rangle_{\sigma}-\sum_{m=0}^{r-1}w^{[m]}(\tau,x)h^{m}+\mathcal{O}(h^{r})\\
&=\sum_{m=0}^{r-1}\big(\sum_{n+k=m}\langle P_{\tau,\sigma}^{[n]},w^{[k]}(\sigma,x)\rangle_{\sigma}-w^{[m]}(\tau,x)\big)h^{m}+\mathcal{O}(h^{r})\\
&=\sum_{m=0}^{r-1}\big(\langle
P_{\tau,\sigma}^{[0]},w^{[m]}(\sigma,x)\rangle_{\sigma}-w^{[m]}(\tau,x)\big)h^{m}+\mathcal{O}(h^{r})=\mathcal{O}(h^{r}).
\end{aligned}
\end{equation*}
 \hfill  $\square$
\end{proof}

Using Lemmas \ref{lemmare} and \ref{lemmaph}, we immediately have
the following result.
\begin{lem}\label{lem new}
For the result $\tilde{u}(\tau,x)$ of  the ECMr method
 \eqref{ECM}, it holds that
\begin{equation*}\begin{aligned}
\mathcal{P}_{h}(f(\tilde{u}(\tau,x)))-f(\tilde{u}(\tau,x))=\mathcal{O}(h^{r}).
\end{aligned}\end{equation*}
\end{lem}

\section{Convergence}\label{sec:convergnce}
Before discussing the convergence of the ECMr methods, we  need the
following assumption and Gronwall's lemma, which are useful for our
analysis.
\begin{assum}\label{assumption}
It is assumed that    $f$ is  Lipschitz-continuous, i.e., there
exists $L
> 0$ such that
\begin{equation*}
\begin{aligned}
&||f (y)-f(z)||_{H^{\alpha}}\leq L||y-z||_{H^{\alpha}}\\
\end{aligned}
\end{equation*}
for all $y,z\in H^{\alpha}(\mathbb{T}^d)$ satisfying $y,z \in
B(\bar{u}_{0},R)$.
\end{assum}

\begin{lem}
\label{Gronwall's lemma}  (Gronwall's lemma) Let $\mu$ be positive
and $a_k, b_k\ (k=0,1,2,\cdots)$ be nonnegative and satisfy
$$a_k\leq(1+\mu\Delta t)a_{k-1}+\Delta tb_k,\quad k=1,2,3,\ldots,$$
then
$$a_k\leq\exp(\mu k\Delta t)\Big(a_0+\Delta t\sum\limits_{m=1}^kb_m\Big),\quad k=1,2,3,\ldots.$$
\end{lem}

In what follows, we omit  $x$ in the expressions for brevity.

Let $e_{n}$  denote the difference between the numerical and exact
solutions at $t_n$, $E_{n,\tau}$  the difference at $t_n+ \tau h$,
 namely,
\begin{equation*}
e_{n}=u_{n}-u(t_{n}),\ E_{n,\tau}=\tilde{u} (t_n+ \tau h )-u (t_n+ \tau h ).\ \label{en}%
\end{equation*}
Then we present the following convergence result.
\begin{theo}\label{thm2}
Under all the assumptions given in this paper, if the time stepsize
$h$ satisfies $0<h<\frac{1}{2M_0L}$, we have
\begin{equation*}
\begin{aligned}
\norm{e_n}_{H^{\alpha}}\leq\exp(2T M_0L)\big(  2 M_0L\bar{C}_1T h
+\bar{C}_2T\big)h^{2r},
\end{aligned}
\end{equation*}
where $L$ is the Lipschitz constant of Assumption \ref{assumption},
$M_0$ is defined in \eqref{MC}, and   $\bar{C}_1$ and $\bar{C}_2$
are independent of $n$
 and $h$.

\end{theo}
\begin{proof}
Inserting the exact solution into the numerical scheme \eqref{ECM}
gives
\begin{equation}\label{so1-PECMr}
\left\{\begin{aligned} u(t_n+\tau h)=&e^{\tau h \mathrm{i}
\mathcal{A}}u(t_{n}) + \mathrm{i}\tau h \int_{0}^1
\bar{A}_{\tau,\sigma}(\mathrm{i} \mathcal{A}) f(u(t_n+\tau\sigma h))
d\sigma+\triangle_{n,\tau},\\
u(t_{n+1})=&e^{ h \mathrm{i} \mathcal{A}}u(t_{n})+\mathrm{i}h
\int_{0}^1 \bar{A}_{1,\sigma}( \mathrm{i} \mathcal{A})
f(u(t_n+\sigma h))
d\sigma+\delta_{n+1}\\
\end{aligned}\right.
\end{equation}
with the defects $\triangle_{n,\tau}$ and $\delta_{n+1}.$ According
to the Duhamel's formula \eqref{Duhamel formu},  we obtain that
\begin{equation}\label{def con}
\begin{aligned}
\triangle_{n,\tau}&=\mathrm{i}\tau h \int_{0}^1
e^{\mathrm{i}(1-\sigma)\tau h \mathcal{A}}\big( f(u(t_n+\tau\sigma
h)) - \mathcal{P}_{h}f(u(t_n+\tau\sigma
h))\big)d\sigma \\
&=\mathrm{i}\tau h \int_{0}^1 \big(
\mathcal{P}_{h}e^{\mathrm{i}(1-\sigma)\tau h
\mathcal{A}}+\mathcal{O}(h^{r})\big)\big( f(u(t_n+\tau\sigma h))
-\mathcal{P}_{h}f(u(t_n+\tau\sigma h))\big)d\sigma\\
&=\mathcal{O}(h^{2r+1}).
\end{aligned}
\end{equation}
In a similar way, one gets
\begin{equation*}
\begin{aligned}
&\delta_{n+1}=\mathrm{i}  h \int_{0}^1 e^{\mathrm{i}(1-\sigma)  h
\mathcal{A}}\big(f(u(t_n+\sigma h)) -\mathcal{P}_{h}f(u(t_n+\sigma
h))\big)d\sigma=\mathcal{O}(h^{2r+1}).
\end{aligned}
\end{equation*}

Subtracting \eqref{so1-PECMr} from \eqref{ECM} leads to the error
recursions
\begin{equation*}\label{err recu}
\begin{aligned} E_{n,\tau}=&e^{\tau h \mathrm{i}
\mathcal{A}}e_{n}+\tau\mathrm{i}h \int_{0}^1 \bar{A}_{\tau,\sigma}(
\mathrm{i} \mathcal{A}) \big[f(\tilde{u}(t_n+\tau h)) -f(u(t_n+\tau
h)) \big]
d\sigma+\triangle_{n,\tau},\\
e_{n+1}=&e^{ h \mathrm{i} \mathcal{A}}e_{n}+\mathrm{i}h \int_{0}^1
\bar{A}_{1,\sigma}( \mathrm{i} \mathcal{A})
\big[f(\tilde{u}(t_n+\tau h))-f(u(t_n+\tau h)) \big]
d\sigma+\delta_{n+1}.
\end{aligned}
\end{equation*}
 We then  have
\begin{equation}\label{err recu-1111}
\begin{aligned} \norm{E_{n,\tau}}_{H^{\alpha}}\leq&\norm{e_{n}}_{H^{\alpha}}+h
M_0L   \norm{E_{n,\tau}}_{H^{\alpha}}+\norm{\triangle_{n,\tau}}_{H^{\alpha}},\\
\norm{e_{n+1}}_{H^{\alpha}}\leq&\norm{e_{n}}_{H^{\alpha}}+  h M_0L
  \norm{E_{n,\tau}}_{H^{\alpha}}+\norm{\delta_{n+1}}_{H^{\alpha}}.\\
\end{aligned}
\end{equation}
If the time stepsize  is chosen by $1-hM_0L\geq\frac{1}{2},$ then
one has the following result
\begin{equation*}
\begin{aligned}  \norm{E_{n,\tau}}_{H^{\alpha}}\leq&2\norm{e_{n}}_{H^{\alpha}}+2 \norm{\triangle_{n,\tau}}_{H^{\alpha}}.\\
\end{aligned}
\end{equation*}
Inserting this into the second inequality of \eqref{err recu-1111}
yields
\begin{equation*}
\begin{aligned}
\norm{e_{n+1}}_{H^{\alpha}}\leq&\norm{e_{n}}_{H^{\alpha}}+  h M_0L
\Big[2\norm{e_{n}}_{H^{\alpha}}+2\norm{\triangle_{n,\tau}}_{H^{\alpha}}\Big]+\norm{\delta_{n+1}}_{H^{\alpha}},\\
\end{aligned}
\end{equation*}
which gives that
\begin{equation*}
\begin{aligned}
 \norm{e_{n+1}}_{H^{\alpha}}\leq&(1+  2h M_0L)
 \norm{e_{n}} _{H^{\alpha}}  +2h M_0L\norm{\triangle_{n,\tau}} _{H^{\alpha}}+\norm{\delta_{n+1}}_{H^{\alpha}}.\\
\end{aligned}
\end{equation*}
Taking into account the fact that
\begin{equation*}
\begin{aligned}
&\sum_{m=1}^{n}\Big(2h M_0L\norm{\triangle_{n,\tau}}_{H^{\alpha}}
+\norm{\delta_{n+1}}_{H^{\alpha}}\Big)\leq\sum_{m=1}^{n}\Big(2h
M_0L\bar{C}_1h^{2r+1}+\bar{C}_2h^{2r+1}\Big)\\
& \leq 2 M_0L\bar{C}_1Th^{2r+1}+\bar{C}_2Th^{2r},
\end{aligned}
\end{equation*}
and  using  Gronwall's lemma, we obtain
\begin{equation*}
\begin{aligned}
\norm{e_n}_{H^{\alpha}}\leq \exp(2TM_0L)\big( 2 M_0L\bar{C}_1T h
+\bar{C}_2T\big)h^{2r}.
\end{aligned}
\end{equation*}
 \hfill  $\square$
\end{proof}

\begin{rem}
From this convergence result, it follows that our exponential
collocation methods can be of arbitrarily high order   by only
choosing a suitable large integer
  $r$, which is  very simple and convenient  in applications.   This feature is significant in the construction
of higher-order   methods.
\end{rem}

\section{Nonlinear
stability}\label{sec:stability}
 This section is devoted to the study
of nonlinear stability. To this end, we consider the following
perturbed problem associated with \eqref{sch ode}
\begin{equation}\label{perturb problem}
\begin{aligned}
&\check{u}_{t}=\mathrm{i}\mathcal{A}\check{u}+\mathrm{i}f(\check{u}),
\ \ \ \ \check{u}(0,x)=u_0(x)+\check{u}_0(x),
\end{aligned}
\end{equation}
where $\check{u}_0(x)\in H^{\alpha}$ is perturbation function.
Letting $ \hat{u}(t,x)=\check{u}(t,x)-u(t,x),$ and subtracting
\eqref{sch ode} from \eqref{perturb problem} yields
\begin{equation}\label{perturb}
\begin{aligned}
\hat{u}_{t}=\mathrm{i}\mathcal{A}\hat{u}+\mathrm{i}f(\check{u})-\mathrm{i}f(u),
\ \ \  \hat u(x,0)=\check{u}_0(x).
\end{aligned}
\end{equation}
Applying   the approximation   respectively to \eqref{sch ode} and
\eqref{perturb problem}, we obtain two numerical schemes, which
leads to an approximation of \eqref{perturb} as follows
\begin{equation}\label{method for perturb}
\left\{\begin{aligned} & \hat  u (t_n+\tau h)=e^{\tau h \mathrm{i}
\mathcal{A}}\hat u_{n}+\tau h \int_{0}^1
\bar{A}_{\tau,\sigma}(\mathrm{i} \mathcal{A}) \mathrm{i}\Big[f(
 \check u (t_n+\tau h))  -f (u (t_n+\tau h))\Big] d\sigma,\\
&\hat u_{n+1}=e^{ h \mathrm{i} \mathcal{A}}\hat u_{n}+ h \int_{0}^1
\bar{A}_{1,\sigma}(\mathrm{i} \mathcal{A}) \mathrm{i}\Big[f(
 \check u (t_n+ h))  -f (u (t_n+ h))\Big] d\sigma,
\end{aligned}\right.
\end{equation}
where $ \hat u_{n+1}=\check u_{n+1}-u_{n+1}.$

\begin{theo}   \label{thm:nonlinear
stability}Under the conditions in Theorem \ref{thm2}, we have the
following nonlinear stability result
\begin{equation*}\label{stability}
\begin{aligned}
\norm{\hat u_n}_{H^{\alpha}}\leq \exp(2T
M_0L)\norm{\check{u}_0(x)}_{H^{\alpha}}.
\end{aligned}
\end{equation*}
\end{theo}
\begin{proof}According to \eqref{method for perturb},  we have
\begin{equation*}\label{err recu-11}
\begin{aligned} \norm{\hat  u (t_n+\tau h)}_{H^{\alpha}}\leq&\norm{\hat u_{n}}_{H^{\alpha}}+  h
M_0L
 \norm{\hat  u (t_n+\tau h)}_{H^{\alpha}} , \\
\norm{\hat u_{n+1}}_{H^{\alpha}}\leq&\norm{\hat u_{n}}_{H^{\alpha}}+
h M_0L \norm{\hat  u (t_n+\tau h)}_{H^{\alpha}} .\\
\end{aligned}
\end{equation*}
From the first result, it follows that $ \norm{\hat  u (t_n+\tau
h)}_{H^{\alpha}}\leq2\norm{\hat u_{n}}_{H^{\alpha}}. $ Inserting
this into the second one leads to
\begin{equation*}
\begin{aligned}
\norm{\hat u_{n+1}}_{H^{\alpha}}\leq&(1+  2h M_0L)\norm{\hat u_{n}}_{H^{\alpha}}.\\
\end{aligned}
\end{equation*}
Therefore, we arrive at
\begin{equation*}
\begin{aligned}
\norm{\hat u_n}_{H^{\alpha}}\leq \exp(2T M_0L)\norm{\hat
u_0(x)}_{H^{\alpha}}=\exp(2T
M_0L)\norm{\check{u}_0(x)}_{H^{\alpha}}.
\end{aligned}
\end{equation*}
 \hfill  $\square$
\end{proof}

\section{Numerical experiments}\label{sec:experiments}
As an example, we choose
$Y=\text{span}^x\left\{\varphi_{0}(t),\varphi_{1}(t),\ldots,\varphi_{r-1}(t)\right\}$
and $
X=\text{span}^x\left\{1,\int_{0}^{t}\varphi_{0}(s)ds,\ldots,\int_{0}^{t}\varphi_{r-1}(s)ds\right\}
$ with $\varphi_{k}(t)=t^k$ for $k=0,1,\ldots,r-1.$   Applying the
$r$-point Gauss--Legendre's quadrature  to the  integral of
\eqref{ECM} yields
\begin{equation}\label{PCRK}
\left\{\begin{aligned} &y_{c_{k}}=e^{c_{k} h \mathrm{i}
\mathcal{A}}y_{0}(x)+c_{k}h \sum_{l=1}^{r}b_l
\bar{A}_{c_{k},c_{l}}(\mathrm{i} \mathcal{A}) \mathrm{i}f(y_{c_{l}}),\quad k=1,2,\ldots,r,\\
&y_{1}(x)=e^{  h \mathrm{i} \mathcal{A}}y_{0}(x)+ h
\sum_{l=1}^{r}b_l \bar{A}_{1,c_{l}}(\mathrm{i} \mathcal{A})
\mathrm{i}f(y_{c_{l}}),
\end{aligned}\right.
\end{equation}
where   $y_{c_{k}}:=\tilde{u}(c_{k},x)$ and
   $c_{l}, b_{l}$ with $k=1,2,\ldots,r$ are the   nodes and weights of the quadrature, respectively.
 We choose $r=2$ and $r=3$ and then denote these two methods by ECM2 and ECM3, respectively.

 In order to  show the efficiency and robustness of these two exponential collocation methods, we
choose another three methods appeared in the literature:
\begin{itemize}
\item  EPC: the energy-preserving collocation fourth order  method  given in
\cite{Hairer2010}, which is precisely the ``extended Labatto IIIA
method of order four" in \cite{Iavernaro2009} if the integral is
approximated  by the Lobatto quadrature of order eight;

\item  EEI:  the explicit exponential integrator of order four derived in
\cite{Hochbruck2009};

\item  EAVF: the energy-preserving exponential AVF method derived in \cite{Li_Wu(sci2016)}.

\end{itemize}

It is  noted  that for all the exponential-type integrators, we use
the Pade approximations to compute  the $\bar{\varphi}$-functions
(see \cite{Berland07soft} for more details). For implicit methods,
we set $10^{-16}$ as the error tolerance and $5$ as the maximum
number of each fixed-point iteration. We also remark that in order
to show that our methods can perform well even for few iterations, a
low maximum number of fixed-point iterations is used.

\vskip2mm\noindent\textbf{Test one.} We first apply these methods to
the nonlinear Schr\"{o}dinger equation \eqref{sch system} with
$\lambda=-2$, $u_0(x)= 0.5 + 0.025 \cos(\mu x)$ and the periodic
boundary condition $\psi(0,t)=\psi(L,t)$. Following \cite{Chen2001},
we consider $L =4\sqrt{2}\pi$, $\mu = 2\pi/L$ and the pseudospectral
method with 128 points. This problem is solved with $T=10,20$ and
 $h= 0.1/2^{i}$ for  $i=2,\ldots,5$. The global errors are presented in Figure \ref{p1}. We also integrate this problem  in  $[0, 100]$
 with $h=1/100$ and $h=1/200$. The  conservation of  discretized energy  is shown in
Figure \ref{p2}. It is noted that when the results are too large for
some methods, we do not plot the corresponding points in the figure.

\vskip2mm\noindent\textbf{Test two.} We then  consider the nonlinear
Schr\"{o}dinger equation \eqref{sch system} with $\lambda=-1$,
$u_0(x)= \frac{1}{1+\sin(x)^2}$ and the same periodic boundary
condition (it has been considered in \cite{Celledoni08}). The global
errors of the intervals $[0, 10]$ and  $[0, 20]$ with the stepsizes
$h= 0.1/2^{i}$ for $i=2,\ldots,5$ are displayed in Figure \ref{p3}.
Figure \ref{p4} indicates the conservation of  discretised energy in
$[0, 100]$ with $h=1/100$ and $h=1/200$ for different methods.

It can be observed from these numerical results that our
 new methods   show a higher accuracy and  a better
 numerical energy-preserving property than the other three methods.

 \begin{figure}[ptb]
\centering
\includegraphics[width=4cm,height=4cm]{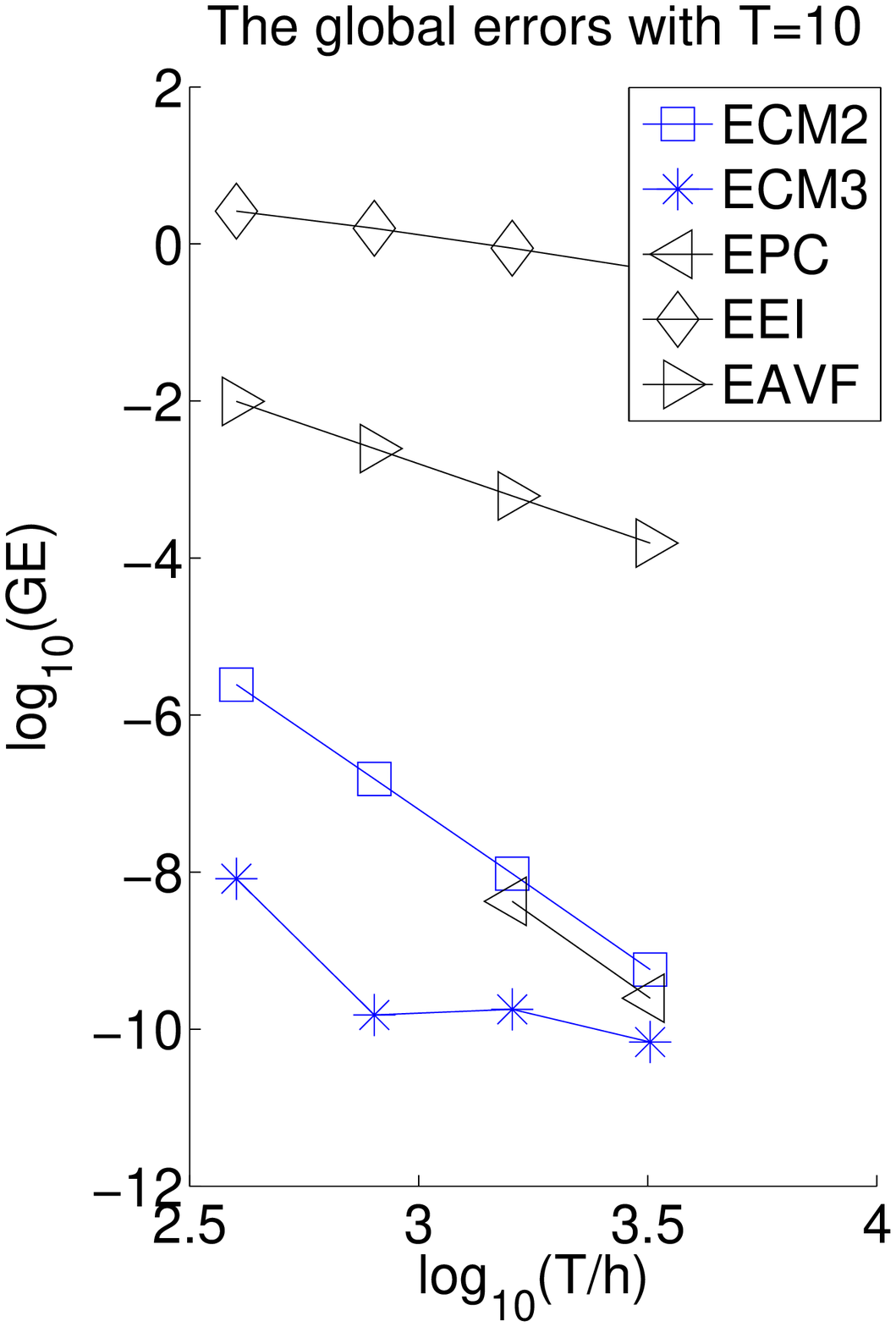}
\includegraphics[width=4cm,height=4cm]{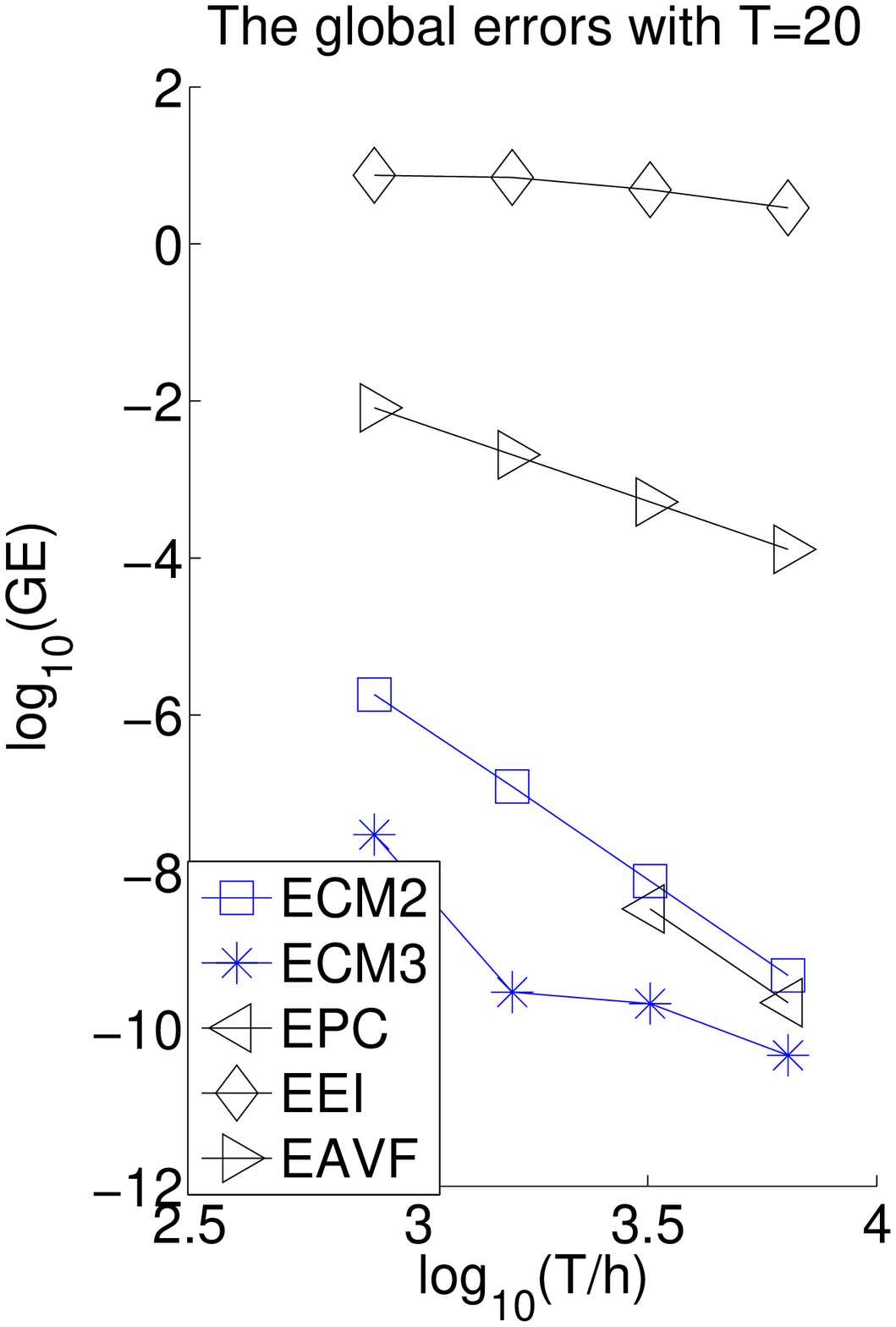}
\caption{The logarithm of the  global error against the logarithm of
$T/h$.} \label{p1}
\end{figure}

 \begin{figure}[ptb]
\centering
\includegraphics[width=8cm,height=3cm]{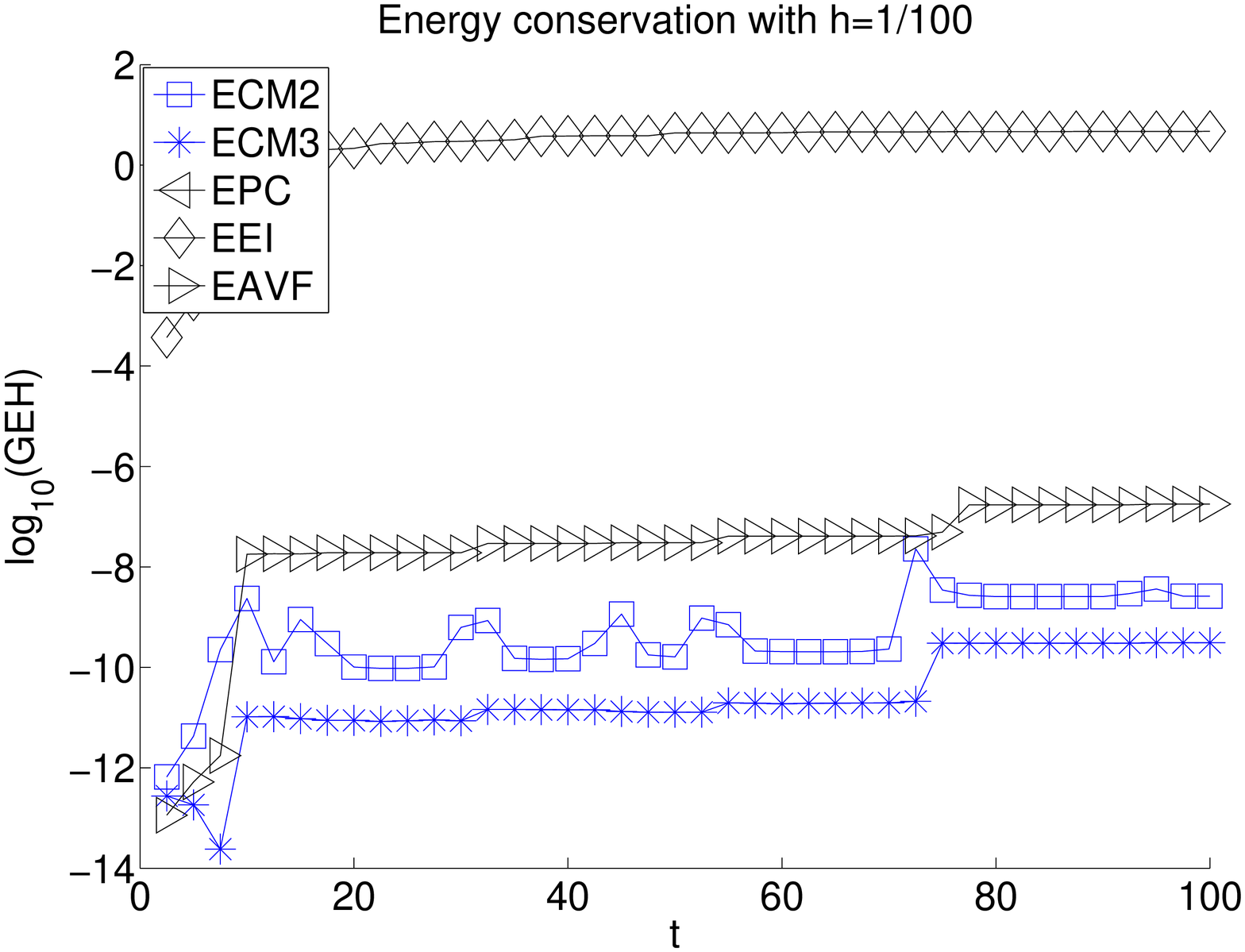}\\
\includegraphics[width=8cm,height=3cm]{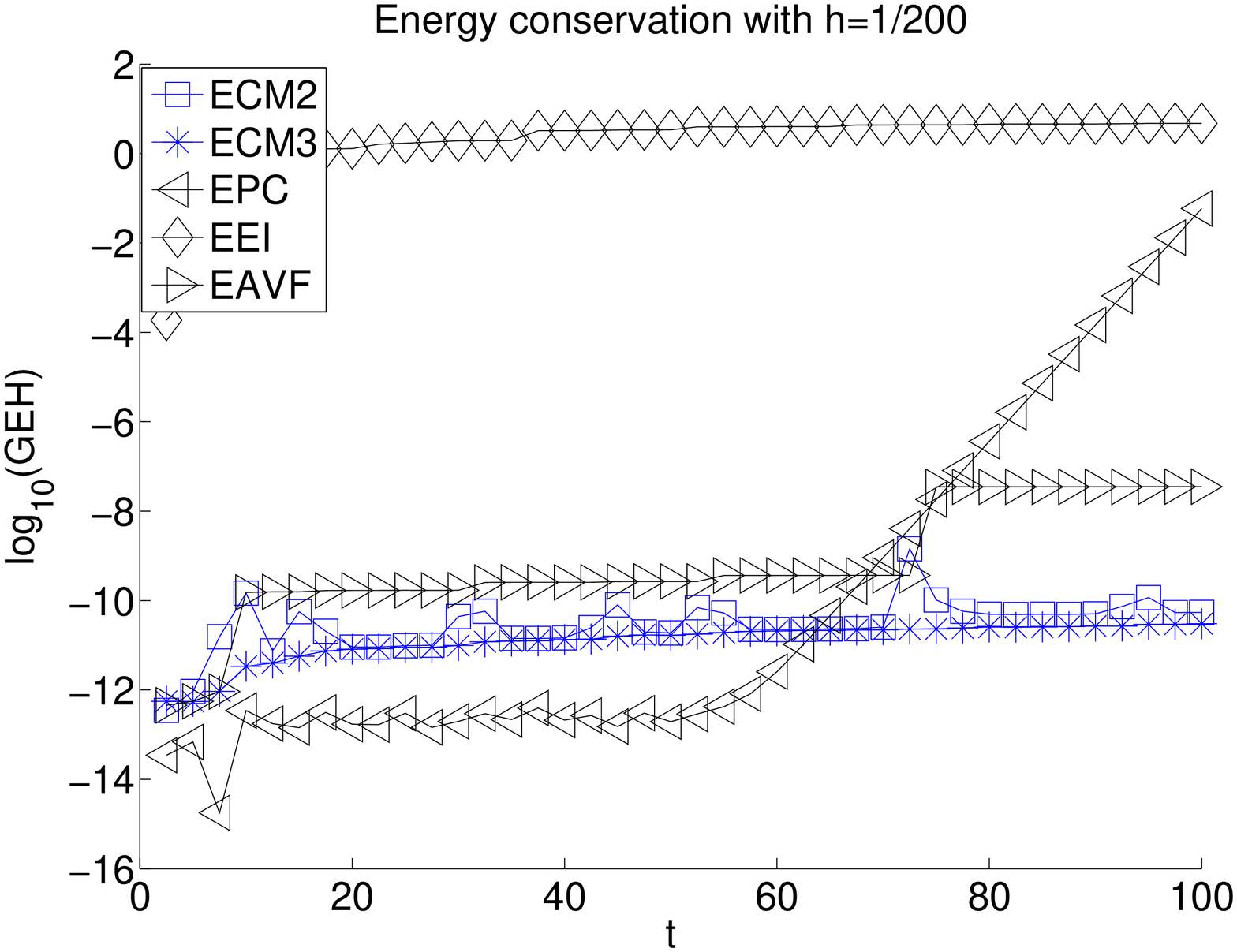}
\caption{The logarithm of the  error of Hamiltonian against  $t$.}
\label{p2}
\end{figure}

 \begin{figure}[ptb]
\centering
\includegraphics[width=4cm,height=4cm]{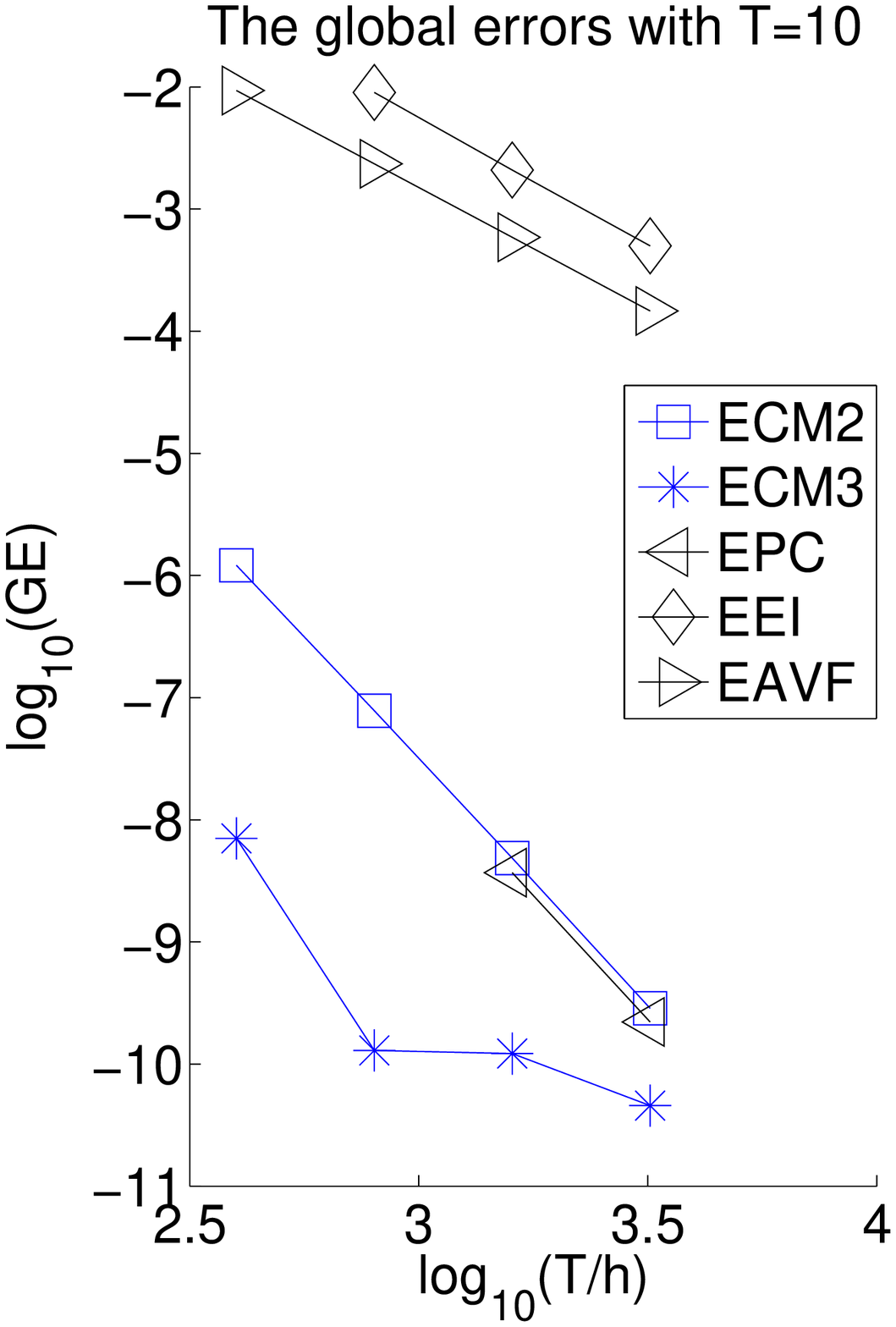}
\includegraphics[width=4cm,height=4cm]{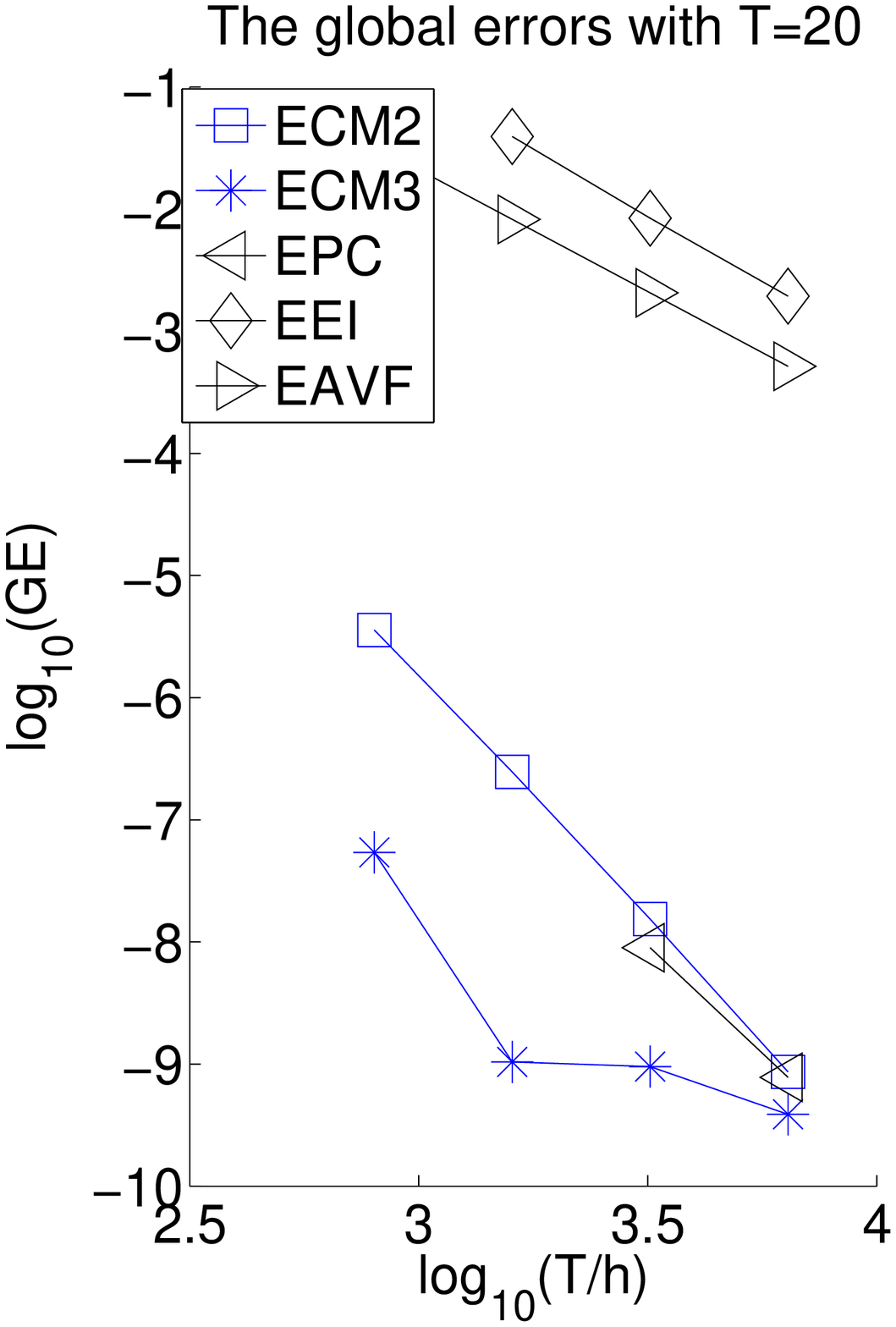}
\caption{The logarithm of the  global error against the logarithm of
$T/h$.} \label{p3}
\end{figure}

 \begin{figure}[ptb]
\centering
\includegraphics[width=8cm,height=3cm]{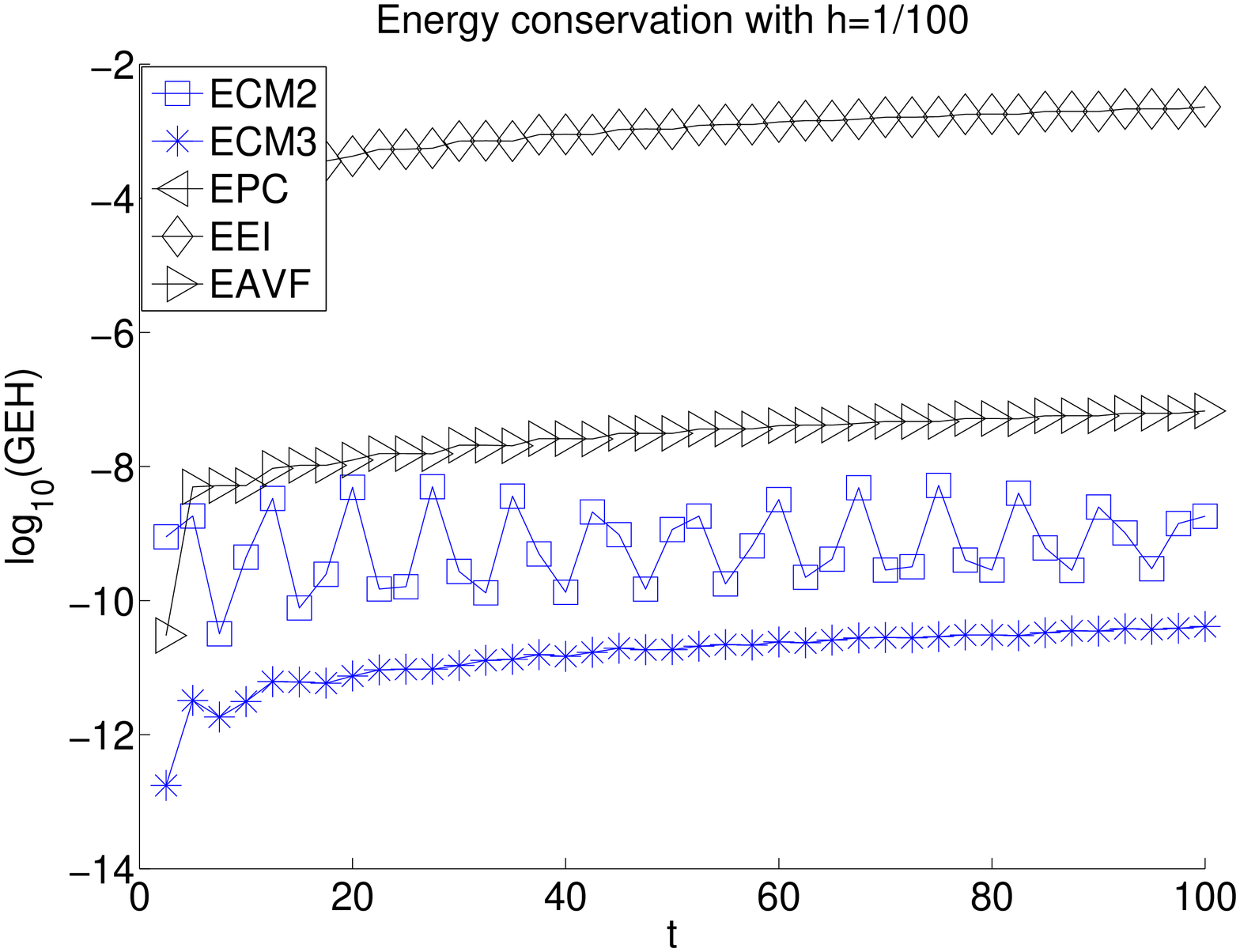}\\
\includegraphics[width=8cm,height=3cm]{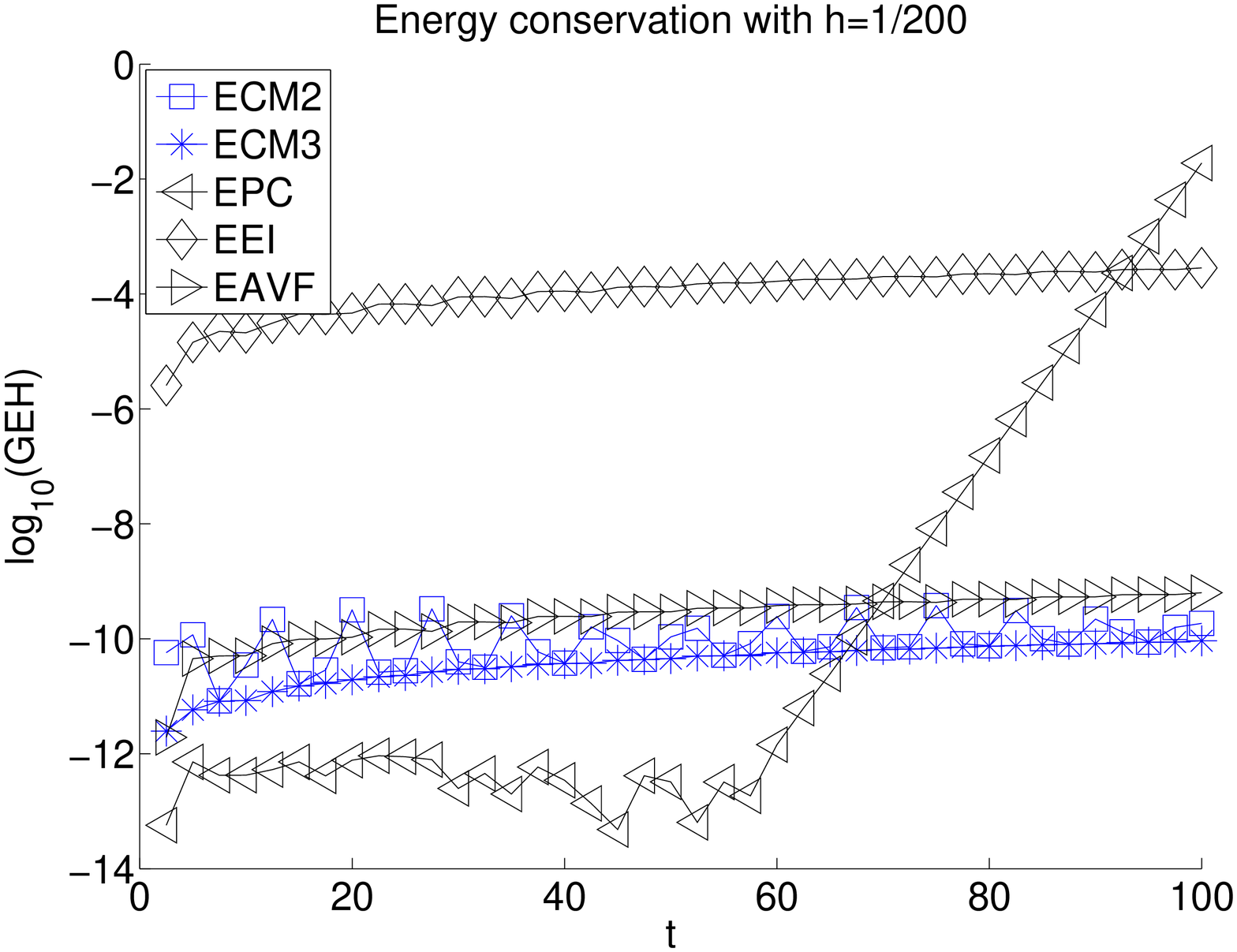}
\caption{The logarithm of the  error of Hamiltonian against  $t$.}
\label{p4}
\end{figure}

\section{Conclusions} \label{sec:conclusions}
We have presented and analysed  exponential collocation methods for
the cubic Schr\"{o}dinge Cauchy problem, which exactly or nearly
preserve   the continuous energy of the original system and can be
of arbitrarily high order. We have also analysed in detail
 the  properties of the  new  methods including existence and uniqueness, global convergence, and nonlinear stability.
Furthermore, the remarkable  efficiency of the    methods  was
demonstrated  by the numerical experiments in comparison with
 existing numerical schemes appeared in the literature.
The application  of the  methodology for the cubic Schr\"{o}dinger
equation \eqref{sch system} stated in this paper to other
Hamiltonian PDEs will be our future work in the near future.

\end{document}